\documentclass[twoside,leqno,10pt, A4]{amsart}
\usepackage{amsfonts}
\usepackage{amsmath}
\usepackage{amscd}
\usepackage{amssymb}
\usepackage{amsthm}
\usepackage{amsrefs}
\usepackage{latexsym}
\usepackage{mathrsfs}
\usepackage{bbm}
\usepackage{amscd}
\usepackage{amssymb}
\usepackage{amsthm}
\usepackage{amsrefs}
\usepackage{latexsym} 
\usepackage{mathrsfs}
\usepackage{bbm}
\usepackage{enumerate}
\usepackage{graphicx}
\usepackage{color}
\begin{document}

\newtheorem{theorem}[subsection]{Theorem}
\newtheorem{proposition}[subsection]{Proposition}
\newtheorem{lemma}[subsection]{Lemma}
\newtheorem{corollary}[subsection]{Corollary}
\newtheorem{conjecture}[subsection]{Conjecture}
\newtheorem{prop}[subsection]{Proposition}
\newtheorem{defin}[subsection]{Definition}

\numberwithin{equation}{section}
\newcommand{\mr}{\ensuremath{\mathbb R}}
\newcommand{\mc}{\ensuremath{\mathbb C}}
\newcommand{\dif}{\mathrm{d}}
\newcommand{\intz}{\mathbb{Z}}
\newcommand{\ratq}{\mathbb{Q}}
\newcommand{\natn}{\mathbb{N}}
\newcommand{\comc}{\mathbb{C}}
\newcommand{\rear}{\mathbb{R}}
\newcommand{\prip}{\mathbb{P}}
\newcommand{\uph}{\mathbb{H}}
\newcommand{\fief}{\mathbb{F}}
\newcommand{\majorarc}{\mathfrak{M}}
\newcommand{\minorarc}{\mathfrak{m}}
\newcommand{\sings}{\mathfrak{S}}
\newcommand{\fA}{\ensuremath{\mathfrak A}}
\newcommand{\mn}{\ensuremath{\mathbb N}}
\newcommand{\mq}{\ensuremath{\mathbb Q}}
\newcommand{\half}{\tfrac{1}{2}}
\newcommand{\f}{f\times \chi}
\newcommand{\summ}{\mathop{{\sum}^{\star}}}
\newcommand{\chiq}{\chi \bmod q}
\newcommand{\chidb}{\chi \bmod db}
\newcommand{\chid}{\chi \bmod d}
\newcommand{\sym}{\text{sym}^2}
\newcommand{\hhalf}{\tfrac{1}{2}}
\newcommand{\sumstar}{\sideset{}{^*}\sum}
\newcommand{\sumprime}{\sideset{}{'}\sum}
\newcommand{\sumprimeprime}{\sideset{}{''}\sum}
\newcommand{\sumflat}{\sideset{}{^\flat}\sum}
\newcommand{\shortmod}{\ensuremath{\negthickspace \negthickspace \negthickspace \pmod}}
\newcommand{\V}{V\left(\frac{nm}{q^2}\right)}
\newcommand{\sumi}{\mathop{{\sum}^{\dagger}}}
\newcommand{\mz}{\ensuremath{\mathbb Z}}
\newcommand{\leg}[2]{\left(\frac{#1}{#2}\right)}
\newcommand{\muK}{\mu_{\omega}}
\newcommand{\thalf}{\tfrac12}
\newcommand{\lp}{\left(}
\newcommand{\rp}{\right)}
\newcommand{\Lam}{\Lambda_{[i]}}
\newcommand{\lam}{\lambda}
\newcommand{\af}{\mathfrak{a}}
\newcommand{\sw}{S_{[i]}(X,Y;\Phi,\Psi)}
\newcommand{\lz}{\left(}
\newcommand{\pz}{\right)}
\newcommand{\bfrac}[2]{\lz\frac{#1}{#2}\pz}
\newcommand{\odd}{\mathrm{\ primary}}
\newcommand{\even}{\text{ even}}
\newcommand{\res}{\mathrm{Res}}
\newcommand{\sumn}{\sumstar_{(c,1+i)=1}  w\left( \frac {N(c)}X \right)}
\newcommand{\lab}{\left|}
\newcommand{\rab}{\right|}
\newcommand{\Go}{\Gamma_{o}}
\newcommand{\Ge}{\Gamma_{e}}
\newcommand{\M}{\widehat}
\def\su#1{\sum_{\substack{#1}}}

\theoremstyle{plain}
\newtheorem{conj}{Conjecture}
\newtheorem{remark}[subsection]{Remark}

\newcommand{\pfrac}[2]{\left(\frac{#1}{#2}\right)}
\newcommand{\pmfrac}[2]{\left(\mfrac{#1}{#2}\right)}
\newcommand{\ptfrac}[2]{\left(\tfrac{#1}{#2}\right)}
\newcommand{\pMatrix}[4]{\left(\begin{matrix}#1 & #2 \\ #3 & #4\end{matrix}\right)}
\newcommand{\ppMatrix}[4]{\left(\!\pMatrix{#1}{#2}{#3}{#4}\!\right)}
\renewcommand{\pmatrix}[4]{\left(\begin{smallmatrix}#1 & #2 \\ #3 & #4\end{smallmatrix}\right)}
\def\en{{\mathbf{\,e}}_n}

\newcommand{\ppmod}[1]{\hspace{-0.15cm}\pmod{#1}}
\newcommand{\ccom}[1]{{\color{red}{Chantal: #1}} }
\newcommand{\acom}[1]{{\color{blue}{Alia: #1}} }
\newcommand{\alexcom}[1]{{\color{green}{Alex: #1}} }
\newcommand{\hcom}[1]{{\color{brown}{Hua: #1}} }

\makeatletter
\def\widebreve{\mathpalette\wide@breve}
\def\wide@breve#1#2{\sbox\z@{$#1#2$}%
     \mathop{\vbox{\m@th\ialign{##\crcr
\kern0.08em\brevefill#1{0.8\wd\z@}\crcr\noalign{\nointerlineskip}%
                    $\hss#1#2\hss$\crcr}}}\limits}
\def\brevefill#1#2{$\m@th\sbox\tw@{$#1($}%
  \hss\resizebox{#2}{\wd\tw@}{\rotatebox[origin=c]{90}{\upshape(}}\hss$}
\makeatletter

\title[Upper bounds for moments of zeta sums]{Upper bounds for moments of zeta sums}

%%\date{\today}

\author[P. Gao]{Peng Gao}
\address{School of Mathematical Sciences, Beihang University, Beijing 100191, China}
\email{penggao@buaa.edu.cn}

\begin{abstract}
We establish upper bounds for moments of zeta sums using results on shifted moments of the Riemann zeta function under the Riemann hypothesis.
\end{abstract}

\maketitle



\section{Introduction}\label{sec 1}

   Character sums have been extensively studied in the literature as they have many important applications in number theory. In \cite{Harper23}, A. J. Harper studied sizes of the sums given by
$$ \sum_{n \leq x} n^{it} \;\;\;\;\; \text{and} \;\;\;\;\; \sum_{n \leq x} \chi(n) , $$    
    where $t \in \mr$ and $\chi(n)$ is a non-principal Dirichlet character modulo a large prime $r$.  Following the notation in \cite{Harper23}, we shall refer the first sum above as a zeta sum.  
    
   Building on his work concerning moments of random multiplicative functions, Harper \cite{Harper23} showed that the low moments of zeta sums (and also character sums) have ``better than squareroot cancellation". More precisely, he proved that uniformly for $1\le x\le T$ and $0\le k\le 1$, 
$$\frac{1}{T}\int_0^T\Big|\sum_{n\le x}n^{i t}\Big|^{2k}\dif t\ll\bigg(\frac{x}{1+(1-k)\sqrt{\log\log (10L_T)}}\bigg)^k,$$
where $L_T=\min\{x,T/x\}.$ 

   In \cite{Szab}, B. Szab\'o obtained sharp upper bounds on shifted moments of Dirichlet $L$-function at points on the critical line and then applied the results to show under the generalized Riemann hypothesis (GRH) that for a fixed real number $k>2$ and a large integer $q$, we have for $2 \leq Y \leq q^{1/2}$, 
\begin{align}
\label{genJacobi}
  \sum_{\chi\in X_q^*}\bigg|\sum_{n\leq Y} \chi(n)\bigg|^{2k} \ll_k \phi(q) Y^k(\log Y)^{(k-1)^2},
\end{align}   
   where $X_q^*$ denotes the set of primitive Dirichlet characters modulo $q$ and $\phi$ denotes Euler's totient function. A similar result is given in \cite{G&Zhao2024} for moments of quadratic Dirichlet character sums under GRH. 
   
  We note that the zeta sums behave very much like character sums. In fact, other than periodicity, the function $n \mapsto n^{-it}$ for a fixed $t \in \mr$ is totally multiplicative and is unimodular.   Thus, one expects to establish results analogous to \eqref{genJacobi} for moments of zeta sums and it is the aim of this paper to achieve this. For this, we define for real numbers $m, T, Y>0$, 
\begin{align*}
%%\label{genJacobi}
  S_m(T, Y):=\int^{2T}_T\bigg|\sum_{n\leq Y} n^{-it}\bigg|^{2m}dt. 
\end{align*} 
 
   We are interested in bounding $S_m(T, Y)$ from the above. We first observe that as pointed out in \cite{Harper23} that using \cite[Lemma 1.2]{Ivic} that when $t$ is large and $x \geq t$, we have $\sum_{x < n \leq 2x} n^{it} = \frac{(2x)^{1+it} - x^{1+it}}{1+it} + O(1) \ll x/t$. Moreover, note that by  \cite[Chap. 7,(34)]{Montgomery94}, we have $\sum_{n \leq x} n^{it} \ll t^{1/2}\log t$ when $x \leq t$. As the term $x/t$ dominates $t^{1/2}\log t$ when $x \geq t^{3/2}\log t$, we deduce that when $T$ is large enough and $Y \geq T^{3/2}\log T$, we have for any real $m>0$, 
\begin{align*}
%%\label{genJacobi}
  S_m(T, Y) \ll_k T^{1-2m} Y^{2m}. 
\end{align*}
 
   We may therefore focus on the case $Y<T^{3/2}\log T$. In fact, we shall assume that $Y \leq (1-\varepsilon)T$ for any $\varepsilon>0$ throughout the paper as this
    is often the most interesting case regarding character sums.  For this case, we  establish the following result concerning the size of $S_m(T, Y)$ under the Riemann 
    hypothesis (RH). 
\begin{theorem}
\label{quadraticmean}
With the notation as above and assume the truth of RH. For any real number $m>2$, large real numbers $T, Y$ such that $Y \leq (1-\varepsilon)T$  for any $\varepsilon>0$, we have
\begin{align}
\label{mainestimation}
 S_m(T,Y)  \ll TY^m(\log T)^{(m-1)^2}.
\end{align}
\end{theorem}

    We note that by H\"older's inequality, we have for any real number $n>1$, 
\begin{align*}
%%\label{SHolder}
 S_m(T,Y) \ll  T^{1-1/n}(S_{mn}(T,Y))^{1/n}. 
\end{align*}
   The above together with Theorem \ref{quadraticmean} then implies that $S_m(T,Y) \ll TY^m(\log T)^{O(1)}$ for any $m>0$, upon choosing $n$ large enough. We remark here that it is shown in \cite{Harper23} that one has $S_m(T,Y) \ll T^{m+1}$, so that our result above improves upon this when $Y$ is slightly smaller than $T$. 

   Our proof of Theorem \ref{quadraticmean} follows the approaches in \cite{Szab}. A key ingredient used in the proof is a result of M. J. Curran \cite{Curran} on shifted moments of the Riemann zeta function $\zeta(s)$.

\section{Preliminaries}
\label{sec 2}

  In this section, we include some results concerning shifted moments of the Riemann zeta function. The first one is quoted from \cite[Theorem 1.1]{Curran}.
\begin{prop}
\label{t1}
 With the notation as above and assume the truth of RH. Let $k\geq 1$ be a fixed integer and $a_1,\ldots, a_{k}$ be fixed non-negative real numbers.
  Let $T$ be a large real number and let ${\bf b}=(b_1,\ldots ,b_{k})$ be a real $k$-tuple with $|b_j|\leq (1-\varepsilon)T$ for a fixed $\varepsilon>0$. Then
\begin{align*}
%%\label{Lprodbounds}
\begin{split}
\int_T^{2T}  \prod_{j = 1}^k |\zeta(\tfrac{1}{2} + i (t + b_k))|^{a_k} dt \ll T (\log T)^{(a_1^2 + \cdots + a_k^2)/4} \prod_{1\leq j < l \leq l} |\zeta(1 + i(b_j - b_l) + 1/ \log T )|^{a_j a_l/2}.
\end{split}
\end{align*}
Here the implied constant depends on $k$ and the $a_j$ but not on $T$ or the $b_j$.
\end{prop}
 
   We remark here that \cite[Theorem 1.1]{Curran} is stated for $|b_j|\leq T/2$ but an inspection of the proof indicates that it continues to hold for 
   $|b_j|\leq (1-\varepsilon)T$ with any $\varepsilon>0$. We also note that 
$$ \Big|\zeta(1+1/\log T+i\alpha)\Big| =\Big|\sum^{\infty}_{n=1}n^{-(1+1/\log T+i\alpha)}\Big | \leq \Big|\sum^{\infty}_{n=1}n^{-(1+1/\log T)}\Big|=\Big|\zeta(1+1/\log T)\Big| \ll \log T,$$
where the last estimation above follows from \cite[Corollary 1.17]{MVa1}. Also by \cite[Corollary 1.17]{MVa1}, we see that for 
$\frac{1}{\log T}\leq |\alpha| \leq 10$, we have
$$ |\zeta(1+1/\log T+i\alpha)|=\frac{1}{|1/\log T+i\alpha|} +O(1) \ll 
\frac{1}{|\alpha|}.$$  
  Moreover, by \cite[Corollary 13.16]{MVa1}, we see that for $10\leq |\alpha|\leq e^T$, we have under the RH that 
$$\log|\zeta(1+1/\log T+i\alpha)|\leq \log\log\log |\alpha|+O(1).$$

  Based on these observations, for $T$ be given as in Proposition \ref{t1}, we now introduce the function $g:\mathbb{R}_{\geq 0} \rightarrow \mathbb{R}$ defined by
\begin{align}
\label{gdef}
\begin{split}
 g(x) =\begin{cases}
\log T  & \text{if } x\leq \frac{1}{\log T} \text{ or } x \geq e^T, \\
\frac{1}{x} & \text{if }   \frac{1}{\log T}\leq x\leq 10, \\
\log \log x & \text{if }  10 \leq x \leq e^{T}.
\end{cases}
\end{split}
\end{align}
 
   The above discussions together with Proposition \ref{t1} allows us to derive the following simplified version on shifted moments of the Riemann zeta function. 
\begin{corollary}
\label{t2}
 With the notation as above and assume the truth of RH. Let $k\geq 1$ be a fixed integer and $a_1,\ldots, a_{k}$ be fixed non-negative real numbers. 
 Let $T$ be a large real number and let ${\bf b}=(b_1,\ldots ,b_{k})$ be a real $k$-tuple with $|b_j|\leq  (1-\varepsilon)T$ for a fixed $\varepsilon>0$. Then
\begin{align*}
%%\label{Lprodbounds}
\begin{split}
\int_T^{2T}  \prod_{j = 1}^k |\zeta(\tfrac{1}{2} + i (t + b_k))|^{a_k} dt \ll T (\log T)^{(a_1^2 + \cdots + a_k^2)/4} \prod_{1\leq j < l \leq l} g(|b_j - b_l|)^{a_j a_l/2}.
\end{split}
\end{align*}
Here the implied constant depends on $k, \varepsilon$ and the $a_j$ but not on $T$ or the $b_j$.
\end{corollary}

   We also note the following upper bounds on moments of the Riemann zeta function, which can be obtained by modifying the proof of \cite[Theorem 1.1]{Curran}.
\begin{lemma}
\label{prop: upperbound}
With the notation as above and assume the truth of RH. Let $k\geq 1$ be a fixed integer and $a_1,\ldots, a_{k}$ be fixed non-negative real numbers. 
Let $T$ be a large real number and let ${\bf b}=(b_1,\ldots ,b_{k})$ be a real $k$-tuple with $|b_j|\leq  (1-\varepsilon)T$ for any fixed $\varepsilon>0$. Then for large real number $T$ and 
$\sigma \geq 1/2$,
\begin{align*}
%%\label{upperbound1}
   \int_T^{2T}  \prod_{j = 1}^k |\zeta(\sigma + i (t + b_k))|^{a_k} dt \ll &  T(\log T)^{O(1)}.
\end{align*}
\end{lemma}

  We end this section by including an estimation for an average of the moments of the Riemann zeta function.
\begin{proposition}
\label{t3prop}
 With the notation as above and assume the truth of RH. We have for any real numbers $m \geq 2$, $10 \leq E \leq (1-\varepsilon)T$ with $\varepsilon>0$ being fixed,
\begin{align}
\label{finiteintest}
\begin{split}
 &\int^{2T}_T\bigg(\int_{0}^{E}|\zeta(1/2+i(\pm s+t))|ds\bigg)^{2m}dt \\
 \ll & T\big( (\log T)^{(m-1)^2}E^3(\log \log E)^{O(1)}+(\log T)^{m^2-3m+3}E^{2m}(\log \log E)^{O(1)}(\log\log T)^{O(1)}\big).
\end{split}
\end{align}
\end{proposition}
\begin{proof}
  Our proof follows closely that of \cite[Proposition 3]{Szab}. Without loss of generality, we prove \eqref{finiteintest} only for the case where the sign $\pm$ in front of $s$  is $+$ in what follows.  We have by symmetry that for each fixed $t$ and any fixed integer $k \geq 1$, 
\begin{align}
\label{Lintdecomp}
    \bigg(\int_{0}^{E} |\zeta(1/2+ i(s+t))| ds\bigg)^{2m}
      \ll \int_{[0,E]^k}\prod_{a=1}^k|\zeta(1/2+ i(t_a+t))| \cdot \bigg(\int_{\mathcal{D} }|\zeta(1/2+i(u+t))|du \bigg)^{2m-k} d\mathbf{t},
\end{align}
where $\mathcal{D}=\mathcal{D}(t_1,\ldots,t_k)=\{ u\in [0,E]:|t_1-u|\leq |t_2-u|\leq \ldots \leq |t_k-u| \}$.

  We let $\mathcal{B}_1=\big[-\frac{1}{\log T},\frac{1}{\log T}\big]$ and $\mathcal{B}_j=\big[-\frac{e^{j-1}}{\log T}, -\frac{e^{j-2}}{\log T}\big]
  \cup \big[\frac{e^{j-2}}{\log T}, \frac{e^{j-1}}{\log T}\big]$ for $2\leq j< \lfloor \log \log T\rfloor+10 := K$. We further denote
  $\mathcal{B}_K=[-E,E]\setminus \bigcup_{1\leq j<K} \mathcal{B}_j$.

Observe that for any $t_1\in [0,E]$,  we have $\mathcal{D}\subset [0,E] \subset t_1+[-E,E]\subset \bigcup_{1\leq j\leq K} t_1+\mathcal{B}_j$. Thus
if we denote $\mathcal{A}_j=\mathcal{B}_j\cap (-t_1+\mathcal{D})$, then $(t_1+\mathcal{A}_j)_{1\leq j\leq K}$ form a partition of $\mathcal{D}$.
We apply Hölder's inequality twice to deduce that for $2m \geq k+1$, 
\begin{align}
\label{LintoverD}
\begin{split}
    & \bigg(\int_{\mathcal{D}}|\zeta(1/2+ i(u+t))|du\bigg)^{2m-k} \\
      \leq & \bigg( \sum_{1\leq j\leq K} \frac{1}{j}\cdot  j \int_{t_1+\mathcal{A}_j} |\zeta(1/2+i(u+t))|du  \bigg)^{2m-k} \\
     \leq & \bigg(\sum_{1\leq j\leq K} j^{2m-k} \bigg( \int_{t_1+\mathcal{A}_j} \big|\zeta(1/2+i(u+t))\big|du  \bigg)^{2m-k}\bigg)
     \bigg(\sum_{1\leq j\leq K } j^{-(2m-k)/(2m-k-1)} \bigg)^{2m-k-1} \\
     \ll & \sum_{1\leq j\leq K} j^{2m-k} \bigg( \int_{t_1+\mathcal{A}_j} |\zeta(1/2+i(u+t))|du \bigg)^{2m-k} \\
     \leq & \sum_{1\leq j\leq K} j^{2m-k} |\mathcal{B}_j|^{2m-k-1} \int_{t_1+\mathcal{A}_j} |\zeta(1/2+i(u+t))|^{2m-k}du.
\end{split}
\end{align}
  We denote for $\mathbf{t}=(t_1,\ldots,t_k)$,
$$\zeta(\mathbf{t},u)=\int^{2T}_T\prod_{a=1}^k|\zeta(1/2+i(t_a+t))| \cdot |\zeta(1/2+i(u+t))|^{2m-k} dt.$$
  We then deduce from \eqref{Lintdecomp} and \eqref{LintoverD} that
\begin{align}
\label{Lintest}
\begin{split}
  \int^{2T}_T\bigg(\int_{0}^{E}|\zeta(1/2+i(s+t))|ds\bigg)^{2m}dt\ll &
    \sum_{1\leq l_0\leq K} l_0^{2m-k} |\mathcal{B}_{l_0}|^{2m-k-1} \int_{[0,E]^k}\int_{t_1+\mathcal{A}_{l_0}} \zeta(\mathbf{t},u)du d\mathbf{t}  \\
     \ll &  \sum_{1\leq l_0, l_1, \ldots l_{k-1}\leq K} l_0^{2m-k} |\mathcal{B}_{l_0}|^{2m-k-1} \int_{\mathcal{C}_{l_0,l_1, \cdots, l_{k-1}}} \zeta(\mathbf{t},u)du d\mathbf{t},
\end{split}
\end{align}
where
$$\mathcal{C}_{l_0,l_1, \cdots, l_{k-1}}=\{(t_1,\ldots,t_k,u)\in [0,E]^{k+1}: u\in t_1+ \mathcal{A}_{l_0},\, |t_{i+1}-u|-|t_i-u|\in \mathcal{B}_{l_i}, \ 1 \leq i \leq k-1\}.$$
We now distinguish two cases in the last summation of \eqref{Lintest} according to the size of $l_0$.

\textbf{Case 1:} $l_0<K$. First note that for any fixed $u$, $t_1$ is in a fixed region of size $\ll \frac{e^{l_0}}{\log T}$. For fixed $u$ and $t_1$, $t_2$ is in a fixed region of size $\ll E\frac{e^{l_1}}{\log T}$ as $|t_2-u| \in |t_1-u|+\mathcal{B}_{l_1}$).  Similar considerations then imply that the volume of the region $\mathcal{C}_{l_0,l_1, \cdots, l_{k-1}}$ is $\ll  E^k\frac{e^{l_0+l_1+\cdots+l_{k-1}} }{(\log T)^k}$. Also,
by the definition of $\mathcal{C}_{l_0,l_1, \cdots, l_{k-1}}$ we have $\frac{e^{l_0}}{\log q}\ll |t_1-u|\ll E =(\log T)^{O(1)}$ so that $g(|t_1- u|)\ll \frac{\log T}{e^{l_0}}\log \log E$, where $g$ is the function defined in \eqref{gdef}. We deduce from the definition of $\mathcal{A}_j$ that $|t_2-u|\geq |t_1-u|$, so that $E \gg  |t_2-u|= |t_1-u|+(|t_2-u|-|t_1-u|)\gg \frac{e^{l_0}}{\log T}+\frac{e^{l_1}}{\log T}$, which implies that $g(|t_2- u|)\ll \frac{\log T}{e^{\max(l_0,l_1) }}\log \log E$.
Similarly, we have  $g(|t_i- u|)\ll \frac{\log T}{e^{\max(l_0,l_1,\ldots, l_{i-1}) }}\log \log E$ for any $1 \leq i \leq k$. 
Moreover, we have $\sum^{j-1}_{s=i}(|t_{s+1}-u|-|t_s-u|) \leq |t_j-t_i| $ for any $1 \leq i < j \leq k$, so that we have $g(|t_{j}- t_i|)\ll \frac{\log T}{e^{\max(l_i,\ldots, l_{j-1} ) } }\log \log E$. We then deduce from Theorem \ref{t1} that for $(t_1,\ldots,t_k,u)\in \mathcal{C}_{l_0,l_1, \cdots, l_{k-1}}$,
\begin{align*}
     & \zeta(\mathbf{t},u) \\
     \ll & T(\log T)^{((2m-k)^2+k)/4}(\log \log E)^{O(1)}
     \bigg(\prod^{k-1}_{i=0}\frac{\log T}{e^{ \max(l_0,l_1,\ldots, l_{i}) }} \bigg)^{(2m-k)/2}
     \bigg(\prod^{k-1}_{i=1} \prod^{k}_{j=i+1}\frac{\log T}{e^{\max(l_i,\ldots, l_{j-1} ) } } \bigg)^{1/2} \\
     = & T(\log T)^{m^2}(\log \log E)^{O(1)}  \exp\Big( -\frac {2m-k}{2}\sum^{k-1}_{i=0}\max(l_0,l_1,\ldots, l_{i})-\frac {1}{2}\sum^{k-1}_{i=1} \sum^{k}_{j=i+1}\max(l_i,\ldots, l_{j-1} )\Big).
\end{align*}
  Here, we adopt the convention throughout the paper that any empty product is defined to be $1$  and any empty sum is defined to be $0$. 
Observe that we have $ |\mathcal{B}_{l_0}|\ll \frac{e^{l_0}}{\log T}$, so that 
\begin{align}
\label{firstcase}
\begin{split}
       &  \sum_{\substack{1\leq l_0<K \\ 1\leq l_1, \ldots l_{k-1}\leq K}}  l_0^{2m-k} |\mathcal{B}_{l_0}|^{2m-k-1} \int_{\mathcal{C}_{l_0,l_1, \cdots, l_{k-1}}} \zeta(\mathbf{t},u)du d\mathbf{t} \\
    \ll & T(\log T)^{(m-1)^2}E^k(\log \log E)^{O(1)}   \\
    &  \cdot \sum_{\substack{1\leq l_0<K \\ 1\leq l_1, \ldots l_{k-1}\leq K}}  l_0^{2m-k}\exp\Big( (2m-k-1)l_0+\sum^{k-1}_{i=0}l_i-\frac {2m-k}{2}\sum^{k-1}_{i=0}\max(l_0,l_1,\ldots, l_{i})-\frac 12\sum^{k-1}_{i=1} \sum^{k}_{j=i+1}\max(l_i,\ldots, l_{j-1} )\Big) \\
    = & T(\log T)^{(m-1)^2}E^k(\log \log E)^{O(1)}  \\
    &  \cdot \sum_{\substack{1\leq l_0<K \\ 1\leq l_1, \ldots l_{k-1}\leq K}}  l_0^{2m-k}
    \exp\Big( \frac {2m-k}{2}l_0+\frac 12\sum^{k-1}_{i=1}l_i-\frac {2m-k}{2}\sum^{k-1}_{i=1}\max(l_0,l_1,\ldots, l_{i})-\frac 12\sum^{k-1}_{i=1} \sum^{k}_{j=i+2}\max(l_i,\ldots, l_{j-1} )\Big).
\end{split}
\end{align}
   We now set $k=4$ to see that in this case deduce from the above that 
\begin{align*}
%%\label{firstcase}
\begin{split}
     &  \frac {2m-k}{2}l_0+\frac 12\sum^{k-1}_{i=1}l_i-\frac {2m-k}{2}\sum^{k-1}_{i=1}\max(l_0,l_1,\ldots, l_{i})-
       \frac 12\sum^{k-1}_{i=1} \sum^{k}_{j=i+2}\max(l_i,\ldots, l_{j-1} ) \\
       \ll &  -(2m-4)\max(l_0,\ldots, l_{k-1} ). 
\end{split}
\end{align*}   
   
     We deduce from \eqref{firstcase} and the above that 
\begin{align}
\label{firstcasesimplified}
\begin{split}
       &  \sum_{\substack{1\leq l_0<K \\ 1\leq l_1, \ldots l_{k-1}\leq K}}  l_0^{2m-k} |\mathcal{B}_{l_0}|^{2m-k-1} \int_{\mathcal{C}_{l_0,l_1, \cdots, l_{k-1}}} \zeta(\mathbf{t},u)du d\mathbf{t} \\
     \ll &   T(\log T)^{(m-1)^2}E^k(\log \log E)^{O(1)} \sum_{\substack{1\leq l_0<K \\ 1\leq l_1, \ldots l_{k-1}\leq K}}  l_0^{2m-k}
     \exp\Big( -(2m-4)\max(l_0,\ldots, l_{k-1} )\Big)  \\
    \ll &  T(\log T)^{(m-1)^2}E^k(\log \log E)^{O(1)},
\end{split}
\end{align}   
      where the last estimation above follows by noting that we have $m>2$. 
     
\textbf{Case 2} $l_0=K$. The volume of the region $\mathcal{C}_{K,l_1, \cdots, l_{k-1}}$ is $\ll E^{k+1}\frac{e^{l_1+\cdots+l_{k-1}} }{(\log T)^{k-1}}$. For each $1 \leq i \leq k$, 
we have $g(|t_i- u|)\ll \log\log  E$.  Also, similar to Case 1, we have $g(|t_j- t_i|) \ll \frac{\log T}{e^{\max(l_i,\ldots, l_{j-1} ) } } $ for $1 \leq i \leq k$.
\begin{align*}
     & \zeta(\mathbf{t},u) \\
     \ll & T(\log T)^{((2m-k)^2+k)/4}(\log \log E)^{O(1)}
     \bigg(\prod^{k-1}_{i=1} \prod^{k}_{j=i+1}\frac{\log T}{e^{\max(l_i,\ldots, l_{j-1} ) } } \bigg)^{1/2} \\
     = & T(\log T)^{((2m-k)^2+k^2)/4}(\log \log E)^{O(1)}  \exp\Big( -\frac 12\sum^{k-1}_{i=1} \sum^{k}_{j=i+1}\max(l_i,\ldots, l_{j-1} )\Big).
\end{align*}
  As $|\mathcal{B}_K|\ll E$, we see that
\begin{align}
\label{secondcase}
\begin{split}
     & \sum_{\substack{1\leq l_1, \ldots l_{k-1}\leq K}}  K^{2m-k} |\mathcal{B}_{K}|^{2m-k-1} \int_{\mathcal{C}_{K,l_1, \cdots, l_{k-1}}} \zeta(\mathbf{t},u)du d\mathbf{t}  \\
     \ll & T(\log T)^{((2m-k)^2+k^2)/4-k+1}E^{2m}(\log \log E)^{O(1)}(\log\log T)^{O(1)} \sum_{\substack{1\leq l_1, \ldots l_{k-1}\leq K}} 
     \exp\Big( \sum^{k-1}_{i=1}l_i-\frac 12\sum^{k-1}_{i=1} \sum^{k}_{j=i+1}\max(l_i,\ldots, l_{j-1} )\Big). 
\end{split}
\end{align}

    We now set $k=4$ to see that in this case deduce from the above that 
\begin{align*}
%%\label{firstcase}
\begin{split}
     &   \sum^{k-1}_{i=1}l_i-\frac 12\sum^{k-1}_{i=1} \sum^{k}_{j=i+1}\max(l_i,\ldots, l_{j-1} ) 
       \ll   l_1/2.
\end{split}
\end{align*}   

   We deduce from \eqref{secondcase} and the above that 
\begin{align}
\label{secondcasesimplified}
\begin{split}
     & \sum_{\substack{1\leq l_1, \ldots l_{k-1}\leq K}}  K^{2m-k} |\mathcal{B}_{K}|^{2m-k-1} 
     \int_{\mathcal{C}_{K,l_1, \cdots, l_{k-1}}} \zeta(\mathbf{t},u)du d\mathbf{t}  \\
     \ll & T(\log T)^{((2m-3)^2+3^2)/4-2}E^{2m}(\log \log E)^{O(1)}(\log\log T)^{O(1)} \sum_{\substack{1\leq l_1, \ldots l_{k-1}\leq K}} 
     \exp\Big( l_1/2 \Big) \\
     \ll &   T(\log T)^{m^2-3m+3}E^{2m}(\log \log E)^{O(1)}(\log\log T)^{O(1)}.
\end{split}
\end{align}
   We now deduce the estimation in \eqref{finiteintest} using \eqref{firstcasesimplified} and \eqref{secondcasesimplified}. This completes the proof of the proposition.
\end{proof}

\section{Proof of Theorem \ref{quadraticmean}}
\label{sec: mainthm}

\subsection{Initial treatments}

   As we explained in the paragraph below Theorem \ref{quadraticmean}, it suffices to establish \eqref{mainestimation}.  We let $\Phi_U(tx)$ be a non-negative smooth function supported on $(0,1)$,
    satisfying $\Phi_U(x)=1$ for $t \in (1/U, 1-1/U)$ with $U$ a parameter to be chosen later and such that
    $\Phi^{(j)}_U(x) \ll_j U^j$ for all integers $j \geq 0$. We denote the Mellin transform of $\Phi_U$ by $\widehat{\Phi}_U$ and we
    observe that repeated integration by parts gives that, for any integer $i \geq 1$ and $\Re(s) \geq 1/2$,
\begin{align}
\label{whatbound}
 \widehat{\Phi}_U(s)  \ll  U^{i-1}(1+|s|)^{-i}.
\end{align}

    We insert the function $\Phi_U(\frac nY)$ into the definition of $S_m(T,Y)$ and apply the triangle inequality to obtain that 
\begin{align}
\label{charintinitial}
\begin{split}
S_m(T,Y) \leq &
\int^{2T}_T\Big | \sum_{n}n^{-it}\Phi_U(\frac {n}{Y})\Big |^{2m}dt
  +\int^{2T}_T\bigg|\sum_{n\leq Y} n^{-it}\big (1-\Phi_U(\frac {n}{Y})\big )\bigg|^{2m}dt.
\end{split}
\end{align} 	

  We further apply the Mellin inversion to obtain that
\begin{align*}
%%\label{charintinitial1}
\begin{split}
\int^{2T}_T\Big | \sum_{n}n^{-it}\Phi_U(\frac {n}{Y})\Big |^{2m}dt=&
\int^{2T}_T\Big | \int\limits_{(2)}\zeta(s+it)Y^s\widehat{\Phi}_U(s)ds\Big |^{2m}dt.
\end{split}
\end{align*} 	

   Observe that by \cite[Corollary 5.20]{iwakow} that under RH, we have for $\Re(s) \geq 1/2$ and any $\varepsilon>0$,
\begin{align}
\label{Lbound}
 \zeta(s) \ll |s|^{\varepsilon}.
\end{align}

   The bounds in  \eqref{whatbound} and \eqref{Lbound} allow us to shift the line of integration in \eqref{charintinitial} to $\Re(s)=1/2$ to obtain that
\begin{align}
\label{charint}
\begin{split}
\int^{2T}_T\Big | \sum_{n}n^{-it}\Phi_U(\frac {n}{Y})\Big |^{2m}dt = &
\int^{2T}_T\Big | \int\limits_{(1/2)}\zeta(s+it)Y^s\widehat{\Phi}_U(s)ds\Big |^{2m}dt.
\end{split}
\end{align} 	

  We split the last integral above according to whether $|\Im(s)| \leq (\log T)^D$ or not for some $D>0$ to be specified later, obtaining
\begin{align*}
%%\label{charintinitial1}
\begin{split}
&\int^{2T}_T\Big | \sum_{n}n^{-it}\Phi_U(\frac {n}{Y})\Big |^{2m}dt \\
\ll &
\int^{2T}_T\Big | \int\limits_{\substack{ (1/2) \\ |\Im(s)| \leq (\log T)^D}}\zeta(s+it)Y^s\widehat{\Phi}_U(s)ds\Big |^{2m}dt+
\int^{2T}_T\Big | \int\limits_{\substack{ (1/2) \\ |\Im(s)| > (\log T)^D}}\zeta(s+it)Y^s\widehat{\Phi}_U(s)ds\Big |^{2m}dt.
\end{split}
\end{align*} 	

   We now set $U=(\log T)^C$ to deduce from \eqref{charintinitial}, \eqref{charint} and the above that  
\begin{align}
\label{Ssimplified}
\begin{split}
S_m(T,Y) \ll &
 \int^{2T}_T\Big | \int\limits_{\substack{ (1/2) \\ |\Im(s)| \leq (\log T)^D}}\zeta(s+it)Y^s\widehat{\Phi}_U(s)ds\Big |^{2m}dt+
\int^{2T}_T\Big | \int\limits_{\substack{ (1/2) \\ |\Im(s)| > (\log T)^D}}\zeta(s+it)Y^s\widehat{\Phi}_U(s)ds\Big |^{2m}dt\\
& +\int^{2T}_T\bigg|\sum_{n\leq Y} n^{-it}\big (1-\Phi_U(\frac {n}{Y})\big )\bigg|^{2m}dt
   \\
   \ll & Y^m\int^{2T}_T
   \Big | \int\limits_{\substack{ (1/2) \\ |s| \leq (\log T)^D}}\Big |\zeta(1/2+i(s+t))
   \Big |\frac 1{1+|s|}ds\Big |^{2m}dt+
\int^{2T}_T\Big | \int\limits_{\substack{ (1/2) \\ |\Im(s)| > (\log T)^D}}\zeta(s+it)Y^s\widehat{\Phi}_U(s)ds\Big |^{2m}dt\\
&+\int^{2T}_T\bigg|\sum_{n\leq Y} n^{-it}\big (1-\Phi_U(\frac {n}{Y})\big )\bigg|^{2m}dt.
\end{split}
\end{align}

  It follows from the above that in order to establish Theorem \ref{quadraticmean}, it remains to prove the following results.
\begin{lemma}
\label{Lsmoothslarge}
With the notation as above and assume the truth of RH. We have for $D \gg C$ large enough and any real number $m >2$, 
\begin{equation}
\label{Lsmoothestslarge}
   \int^{2T}_T\Big | \int\limits_{\substack{ (1/2) \\ |\Im(s)| > (\log T)^D}}\zeta(s+it)Y^s\widehat{\Phi}_U(s)ds\Big |^{2m}dt \ll TY^m.
\end{equation}
\end{lemma}

\begin{lemma}
\label{Lsmooth}
With the notation as above and assume the truth of RH. We have for $D$ large enough and any real number $m>2$, 
\begin{equation}
\label{Lsmoothest}
\int^{2T}_T
   \Big | \int\limits_{\substack{ (1/2) \\ |s| \leq (\log T)^D}}\Big |\zeta(1/2+i(s+t))\Big |\frac 1{1+|s|}ds\Big |^{2m}dt \ll T(\log T)^{(m-1)^2}.
\end{equation}
\end{lemma}

\begin{lemma}
\label{fdiff}
With the notation as above and assume the truth of RH. We have for $C$ large enough and any real number $m>2$,  
\begin{equation}
\label{theorem3firstrest}
\int^{2T}_T\bigg|\sum_{n\leq Y} n^{-it}\big (1-\Phi_U(\frac {n}{Y})\big )\bigg|^{2m}dt \ll TY^m.
\end{equation}
\end{lemma}

\subsection{Proof of Lemma \ref{Lsmoothslarge}}

  We apply \eqref{whatbound} and H\"older's inequality to deduce that, as $m \geq 1/2$, 
\begin{align}
\label{LintImslarge}
\begin{split}
 &\int^{2T}_T\Big | \int\limits_{\substack{ (1/2) \\ |\Im(s)| > (\log T)^D}}\zeta(s+it)Y^s\widehat{\Phi}_U(s)ds\Big |^{2m}dt\\
  \ll &\int^{2T}_T\Big ( \Big ( \int\limits_{\substack{ (1/2) \\ |\Im(s)| > (\log T)^D}}\Big | \widehat{\Phi}_U(s) \Big| |\dif s| \Big )^{2m-1} \cdot
  \int\limits_{\substack{ (1/2) \\ |\Im(s)| > (\log T)^D}}\Big |\zeta(s+it)Y^s \Big|^{2m} \Big| \widehat{\Phi}_U(s)\Big | |\dif s|\Big ) dt  \\
 \ll & Y^m\Big (\int\limits_{\substack{ (1/2) \\ |\Im(s)| > (\log T)^D}}\Big | \widehat{\Phi}_U(s) \Big| |\dif s| \Big )^{2m-1} \cdot
 \int\limits_{\substack{ (1/2) \\ |\Im(s)| > (\log T)^D}}\int^{2T}_T\Big |\zeta(s+it)\Big|^{2m} \Big| \widehat{\Phi}_U(s)\Big | dt|\dif s|.
\end{split}
\end{align} 	

   Now we note that by \eqref{whatbound},
\begin{align}
\label{phiint}
\begin{split}
  \int\limits_{\substack {(1/2) \\ |\Im(s)| > (\log T)^D}}\Big | \widehat{\Phi}_U(s) \Big| |\dif s|
 \ll & \int\limits_{\substack {(1/2) \\ |\Im(s)| > (\log T)^D}}\frac U{1+|s|^2} |\dif s| \ll_D \frac {U}{(\log T)^D}.
\end{split}
\end{align}

  We also apply \eqref{Lbound} to see that
\begin{align}
\label{Lintslarge}
\begin{split}
  & \int\limits_{\substack {(1/2) \\ |\Im(s)| > (\log T)^D}}\int^{2T}_T
  \Big | \zeta(s+it)\Big |^{2m} \cdot \Big |\widehat{\Phi}_U(s)\Big | dt|\dif s| \\
  \ll &  \int\limits_{\substack {(1/2) \\ |s| > (\log T)^D}}\int^{2T}_T\Big | \zeta(1/2+i(s+t))\Big |^{2m}\cdot\frac U{1+|s|^2}  dt|\dif s| \\
  \ll & \int\limits_{\substack {(1/2) \\ (\log T)^D< |s| \leq  5T}}\int^{2T}_T\Big | \zeta(1/2+i(s+t))\Big |^{2m}\cdot\frac U{1+|s|^2}  dt|\dif s|+\int\limits_{\substack {(1/2) \\  |s| >  5T}}\int^{2T}_T\Big | \zeta(1/2+i(s+t))\Big |^{2m}\cdot\frac U{1+|s|^2}  dt|\dif s| \\
   \ll & \int\limits_{\substack {(1/2) \\ (\log T)^D< |s| \leq  5T}}\int^{2T}_T\Big | \zeta(1/2+i(s+t))\Big |^{2m}\cdot\frac U{1+|s|^2}  dt|\dif s|+\int\limits_{\substack {(1/2) \\  |s| >  5T}}\int^{2T}_T \Big |s+t\Big |^{\varepsilon}\cdot\frac U{1+|s|^2}  dt|\dif s| \\
   \ll & \int\limits_{\substack {(1/2) \\ (\log T)^D< |s| \leq  5T}}\int^{2T}_T\Big | \zeta(1/2+i(s+t))\Big |^{2m}\cdot\frac U{1+|s|^2} dt |\dif s|+\int\limits_{\substack {(1/2) \\  |s| >  5T}}\int^{2T}_T \Big |s\Big |^{\varepsilon}\cdot\frac U{1+|s|^2} dt |\dif s| \\
   \ll & \int\limits_{\substack {(1/2) \\ (\log T)^D< |s| \leq  5T}}\int^{2T}_T\Big | \zeta(1/2+i(s+t))\Big |^{2m}\cdot\frac U{1+|s|^2} dt |\dif s|+U T^{\varepsilon} \\
    \ll & \int\limits_{\substack {(1/2) \\ (\log T)^D< |s| \leq  5T}}\int^{10T}_{-10T}\Big | \zeta(1/2+it)\Big |^{2m}\cdot\frac U{1+|s|^2} |\dif t| |\dif s|+U T^{\varepsilon} \\
    \ll & TU(\log T)^{O(1)}(\log T)^{-D}+U T^{\varepsilon},
\end{split}
\end{align}
   where the last estimation above follows from \cite[Corollary B]{Sound2009}, which asserts that  
\begin{align*}
%%\label{zetaint}
\begin{split}
 \int^{10T}_{-10T}\Big | \zeta(1/2+it)\Big |^{2m}||\dif t| \ll  T(\log T)^{O(1)}.
\end{split}
\end{align*}    

   We now deduce the estimation given in \eqref{Lsmoothestslarge} from \eqref{LintImslarge}-\eqref{Lintslarge}, upon taking $D \gg C$ large enough. This completes the proof of the lemma. 

\subsection{Proof of Lemma \ref{Lsmooth}}

  We deduce from \eqref{Ssimplified} by symmetry and H\"older's
inequality that,
\begin{align*}
%%\label{Lintsmall}
\begin{split}
  & \Big | \int\limits_{\substack{ (1/2) \\ |s| \leq (\log T)^D}}\Big |\zeta(1/2+i(s+t))|\frac 1{1+|s|}ds\Big |^{2m} \\
  \ll & \Big |\int_0^{(\log T)^D} \frac{|\zeta(1/2+i(s+t))|}{t+1}dt\Big |^{2m} \\
   \leq & \bigg(\sum_{n\leq D\log \log T+1} n^{-2m/(2m-1)} \bigg)^{2m-1}
    \sum_{n\leq  D\log \log T+1} \bigg(n\int_{e^{n-1}-1}^{e^{n}-1 } \frac{|\zeta(1/2+i(s+t)) |}{s+1} ds\bigg)^{2m}   \\
    \ll & \sum_{n\leq  D\log \log T+1} \frac{n^{2m} }{e^{2nm} } \bigg( \int_{e^{n-1}-1}^{e^{n}-1 } |\zeta(1/2+i(s+t)) | ds \bigg)^{2m}.
\end{split}
\end{align*}

   We apply Proposition \ref{t3prop} to see that for any integer $k \geq 1$ and any real numbers $2m \geq k+1, \varepsilon>0$, 
\begin{align*}
  & \sum_{n\leq D\log \log T+1} \frac{n^{2m} }{e^{2nm} }\int^{2T}_T \bigg( \int_{e^{n-1}-1}^{e^{n}-1 } |\zeta(1/2+i(s+t)) | ds \bigg)^{2m}dt
    \\
     \ll & T\sum_{n\leq D\log \log T+1} \frac{n^{2m} }{e^{2nm} }
     \Big((\log T)^{(m-1)^2}e^{kn}(\log 2n)^{O(1)}+(\log T)^{m^2-3m+3}(\log 2n)^{O(1)}(\log\log T)^{O(1)} e^{2mn}\Big) \\
  \ll & T(\log T)^{(m-1)^2}.
\end{align*}
   We now deduce from the above
 that \eqref{Lsmoothest} holds. This completes the proof of the lemma.

\subsection{Proof of Lemma \ref{fdiff}}  We apply the Cauchy-Schwarz inequality to see that
\begin{align}
\label{pocs1}
\begin{split}
   &\int^{2T}_T\bigg|\sum_{n\leq Y} n^{-it}\big (1-\Phi_U(\frac {n}{Y})\big )\bigg|^{2m}dt \\
    \leq & \bigg(\int^{2T}_T\bigg|\sum_{n\leq Y} n^{-it}\big (1-\Phi_U(\frac {n}{Y})\big )\bigg|^{2}dt\bigg)^{1/2}
    \bigg(\int^{2T}_T\bigg|\sum_{n\leq Y} n^{-it}\big (1-\Phi_U(\frac {n}{Y})\big )\bigg|^{4m-2}dt\bigg)^{1/2}.
\end{split}
\end{align}
   We first note that it follows from \cite[Theorem 9.1]{iwakow} that for arbitrary complex numbers $a_n$, we have for $T, Z \geq 2$ and any
   $\varepsilon>0$,
\begin{align*}
%%\label{largesievesquarefree}
\begin{split}
  &\int^{2T}_T\bigg|\sum_{n\leq Z} a_nn^{-it}\bigg|^{2} dt \ll  (T+Z)\sum_{\substack{n \leq Z}}|a_n|^2.
\end{split}
\end{align*}

   We apply the above to $Z=Y$ and keep in mind our assumption that $Y \leq T$ to see that
\begin{align}
\label{pocs2}
\begin{split}
  &\int^{2T}_T\bigg|\sum_{n\leq Y} n^{-it}\big (1-\Phi_U(\frac {n}{Y})\big )\bigg|^{2}dt \ll  (T+Y)\sum_{\substack{n \leq Y } }\big (1-\Phi_U(\frac {n}{Y})\big )^2 
  \ll 
  T\sum_{\substack{Y(1-1/U) \leq n\leq Y} }1+T\sum_{\substack{0 \leq n\leq Y/U} }1 \ll \frac {TY}{U}.
\end{split}
\end{align}

 We next note that
\begin{equation}
\label{pocs4}
\int^{2T}_T\bigg|\sum_{n\leq Y} n^{-it}\big (1-\Phi_U(\frac {n}{Y})\big )\bigg|^{4m-2}dt \ll\int^{2T}_T\bigg|\sum_{n\leq Y} n^{-it}\bigg|^{4m-2}dt+\int^{2T}_T\bigg|\sum_{n\leq Y} n^{-it}\Phi_U(\frac {n}{Y})\bigg|^{4m-2}dt.
\end{equation}

  By the remark made in the paragraph below Theorem \ref{quadraticmean}, we see that
\begin{equation}
\label{pocs6}
\int^{2T}_T\bigg|\sum_{n\leq Y} n^{-it}\Phi_U(\frac {n}{Y})\bigg|^{4m-2}dt \ll TY^{2m-1}(\log T)^{O(1)}.
\end{equation}  

 To estimate the first expression on the right-hand side of \eqref{pocs4}, we apply Perron's formula as given in \cite[Corollary 5.3]{MVa1} to see that
\begin{align}
\label{Perron}
\begin{split}
    \sum_{n\leq Y}n^{-it}= & \frac 1{ 2\pi i}\int_{1+1/\log Y-iY}^{1+1/\log Y+iY}\zeta(s+it) \frac{Y^s}{s}ds +R_1+R_2, \\
     = &\frac 1{ 2\pi i}\int_{1+1/\log Y -iY}^{1/2-iY} +\frac 1{ 2\pi i}\int_{1/2-iY}^{1/2+iY}+\frac 1{ 2\pi i}\int_{1/2+iY}^{1+1/\log Y+iY} \zeta(s+it)\frac{Y^s}{s}ds +R_1+R_2,
\end{split}
\end{align}
  where
\begin{align}
\label{R12}
\begin{split}
  R_1=& O\Big (\sum_{\substack{Y/2< n <2Y  \\  n \neq Y  }}\min (1, \frac {1}{|n-Y|}) \Big )=O(\log Y), \\
   R_2=& O\Big (\frac {4^{1+1/\log Y}+Y^{1+1/\log Y}}{Y}\zeta(1+1/\log Y)\Big )=O(\log Y).
\end{split}
\end{align}
  Here the last estimation above follows from \cite[Corollary 1.17]{MVa1}.
We now consider the moments of the horizontal integrals in \eqref{Perron}. We may assume that $Y\geq 10$, otherwise the lemma is trivial. 
By symmetry we only need to consider only one of them. Note that we have $|Y^s/s|\ll 1$ in that range and $m \geq 1$, which allows us to apply Hölder's inequality to get
\begin{align}
\label{horizontalint}
\begin{split}
\int^{2T}_T\bigg| \int_{1/2+iY}^{1+1/\log Y+iY} \zeta(s+it)\frac{Y^s}{s}ds\bigg|^{4m-2} dt\ll & \int^{2T}_T\bigg( \int_{1/2+iY}^{1+1/\log Y+iY} |\zeta(s+it)| |\dif s| \bigg)^{4m-2}dt \\
 \ll &\int_{1/2+iY}^{1+1/\log Y+iY} \int^{2T}_T|\zeta(s+it)|^{4m-2} dt|\dif s| \\
 \ll & T(\log T)^{O(1)},
\end{split}
\end{align}
  where the last estimation above follows from Lemma \ref{prop: upperbound}, which implies that for $ 1/2 \leq \Re(s) \leq 1+1/\log Y$, we have under RH, 
$$\int^{2T}_T|\zeta(s+it)|^{4m-2}dt \ll T(\log T)^{O(1)}.$$

  We treat the moments of the vertical integral in \eqref{Perron} using  Hölder's inequality (by noting that $4m-2>4$), 
  Proposition \ref{t3prop} and the assumption $Y \leq (1-\varepsilon)T$ to see that
\begin{align}
\label{verticalint}
\begin{split}
    &\int^{2T}_T\bigg|\int_{1/2-iY}^{1/2+iY}\zeta(s+it) \frac{Y^s}{s}ds\bigg|^{4m-2}dt \\
  \ll & Y^{2m-1}\int^{2T}_T \bigg( \int_{0}^Y \frac{|\zeta(1/2+i(s+t)) |}{s+1} ds \bigg)^{4m-2}dt  \\
   \ll &  Y^{2m-1}\sum_{n\leq \log Y+2} \frac{n^{4m-2} }{e^{(4m-2)n}}\int^{2T}_T \bigg( \int_{e^{n-1}-1}^{e^{n}-1 } |\zeta(1/2+i(s+t)) | ds \bigg)^{4m-2}dt \\
  \ll & Y^{2m-1}T(\log T)^{O(1)}\Big ( \sum_{n\leq \log Y+2}\frac{n^{4m-2}}{e^{(4m-2)n} }e^n +\sum_{n\leq \log Y+2}n^{4m-2} \Big )\\
  \ll & Y^{2m-1}T(\log T)^{O(1)}.
\end{split}
\end{align}
  We conclude from \eqref{Perron}-\eqref{verticalint} that 
\begin{equation}
\label{pocs10}
 \int^{2T}_T\bigg|\sum_{n\leq Y} n^{-it}\bigg|^{4m-2}dt \ll Y^{2m-1}T(\log T)^{O(1)}.
\end{equation}
 We then deduce from \eqref{pocs1}-\eqref{pocs6}, \eqref{pocs10} and recall that we have $U=(\log T)^{C}, Y \leq (1-\varepsilon)T$ to see that the estimation given in \eqref{theorem3firstrest} is valid.  This completes the proof of the lemma.

\vspace*{.5cm}

\noindent{\bf Acknowledgments.}  The author is supported in part by NSFC grant 11871082. This work grows out of discussions with Changhao Chen and Nankun Hong on 
large inequalities for Dirichlet polynomials when the author visited Anhui University in April 2024. The author is debt to them for the inspiration of this paper and many helpful suggestions on the writing of the manuscript.

\bibliography{biblio}
\bibliographystyle{amsxport}

\end{document}